\documentclass{article}
\usepackage{graphicx, amsmath, amsthm, amssymb, enumerate,adjustbox, 
float, parskip, bm, subcaption, cite, listings} 
\usepackage[shortlabels]{enumitem}
\usepackage[colorlinks=true, linkcolor=blue, citecolor=blue, urlcolor=blue]{hyperref}

\title{Explicit Consistency Error Estimate for Finite Element Solutions of the Poisson Equation on Convex Domains}
\author{Su Ruibo}
\date{October 2025}

\begin{document}

\maketitle

\begin{abstract}
We derive explicit a priori consistency error estimates for a standard finite element discretization of the Poisson equation on convex domains, where the domain is approximated by an internal convex polyhedron. The obtained explicit estimates depend only on global geometric parameters and are applicable to general convex domains and arbitrary families of simplicial meshes.
\end{abstract}

\newcommand*\Laplace{\Delta}

\newcommand{\Domain}{\Omega}
\newcommand{\Boundary}{\partial \Domain}
\newcommand{\EucSp}{\mathbb{R}^n}
\newcommand{\SubDomain}{\Omega_{\delta}}
\newcommand{\SubBoundary}{\partial\SubDomain}
\newcommand{\Triangulation}{\mathcal{T}_{h}}
\newcommand{\facets}{\mathcal{F}_{\delta}}
\newcommand{\maxForce}{\|f\|_{L^\infty(\Domain)}}

\newtheorem{theorem}{Theorem}[section]
\newtheorem{lemma}[theorem]{Lemma}
\newtheorem{proposition}[theorem]{Proposition}
\newtheorem{definition}[theorem]{Definition}
\newtheorem{corollary}[theorem]{Corollary}
\newtheorem{Remark}[theorem]{Remark}
\newtheorem{Definition}[theorem]{Definition}

\section{Introduction}
The Dirichlet problem of the Poisson equation is given as
\begin{equation}\nonumber
    \left\{\begin{aligned}
    -\Laplace \phi &= f \text{ in }\Domain,\\
    \phi &= g\text{ on }\Boundary,
    \end{aligned}\right.
\end{equation}
where $\Domain$ is an open, bounded set in $\EucSp$.

A standard approach to the Dirichlet problem is to decompose it into two subproblems: a Laplace equation with the prescribed boundary value $g$, and a Poisson equation with homogeneous Dirichlet boundary conditions,
\begin{equation}\label{ExactProblem}
    \left\{\begin{aligned}
    -\Laplace u &= f \text{ in }\Domain,\\
    u &= 0\text{ on }\Boundary,
    \end{aligned}\right.
\end{equation}
The harmonic part, satisfying the Laplace equation, can be approximated using potential methods (see, for example,~\cite{chargeMethod}) or other techniques, for which a supremum error estimate follows from the maximum principle. We focus on the second part, the Poisson equation with homogeneous boundary conditions, and apply the Finite Element Method to it. 

For 2-dimensional polygonal domains, Liu\cite{doi:10.1137/120878446, Liu2024} proposed computable priori error estimates that require either evaluation of local interpolation error estimates on each element for $H^2$ solutions or solving two finite element problems using the Lagrange FEM and the Raviart–Thomas FEM for  solutions without $H^2$-regularity. His result was based on his earlier studies on interpolation error on triangular elements\cite{KIKUCHI20073750}, and has inspired further studies on local error estimate\cite{NAKANO2023115061} and non-homogeneous Neumann problem\cite{li2018explicitfiniteelementerror}. Error analysis for general boundary conditions on smooth domains has also been investigated in \cite{Barrett1986,chiba2022nitschesmethodrobinboundary,kashiwabara2023finiteelementanalysisgeneralized}, although those results involve unspecified constants.

In this paper, we consider a simple yet realistic computational setting where both the domain and the source term are perturbed, and derive an explicit consistency error estimate for \eqref{ExactProblem} on convex domains in dimensions $n=2,3$. The obtained bounds are fully explicit and depend only on global geometric parameters and seminorms of known functions, and can be further refined by case-specific calculations. In the two-dimensional case, we further propose explicit consistency error estimates that do not involve the minimal angle of the mesh. These results do not require mesh-specific computation and are therefore applicable to general families of meshes. 

The derivation proceeds as follows. First, the domain is approximated by a polyhedral domain, and the resulting boundary perturbation is analyzed using a barrier function argument. Next, based on established local interpolation error estimates \cite{KIKUCHI20073750, kobayashi2025, Arcangeli,TetrahedronExplicit}, we obtain an explicit bound for the error between the weak solution in the perturbed domain and its FEM approximation.

\section{Preliminaries}

\subsection{Domain and Function Spaces}

Let the domain $\Domain$ be convex and bounded , and we denote $D=\text{diam}(\Domain)$, and $|U|$ as the Lebesgue measure of any measurable set $U$. The source term $f$ is assumed to belong to $L^2\cap L^\infty$, and is sufficiently smooth for the subsequent discussion. 

We denote by $|\cdot|_k$ the standard $H^k(\Domain)$ seminorms. For clarity, the $H^2$ seminorm is given by 
\begin{equation}
    |u|_2 =\big(\int_{\Domain}\sum_{i,j=1}^{n}|\partial_i \partial_j u|^2 dx\big)^{1/2}.
\end{equation}

For any open domain $U$, each function $a\in H_{0}^1(U)$ is extended by 0 out of $U$.

$u$ is a weak solution of \eqref{ExactProblem} if and only if
\begin{equation}
    \int_{\Domain}\nabla u\cdot \nabla v dx=\int_{\Domain} f v dx, \forall v\in H_0^1(\Domain)
\end{equation}

\subsection{Approximation of Domain}
We approximate the domain $\Domain$ with a convex open polyhedron $\SubDomain\subset \overline{\Domain}$ whose vertices lie on $\Boundary$. The gap $\delta$ between $\SubBoundary$ and $\Boundary$ is defined as follows.

\begin{figure}[H]
    \centering
    \includegraphics[width=0.4\linewidth]{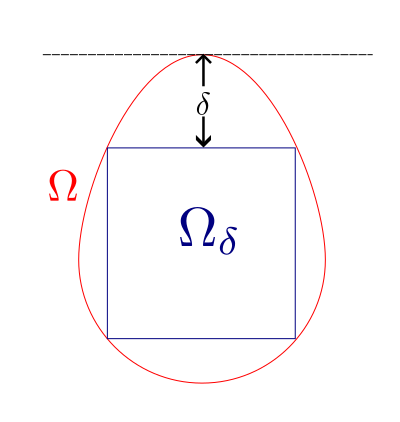}
    \caption{gap width}
    \label{fig:defGap}
\end{figure}

For each facet $F$ of $\SubDomain$, we denote the outward unit normal vector as $\vec{n}_{F}$, and the barycenter as $\vec{g}_F$. We denote  the set of all boundary facets of $\SubDomain$ as $\facets$.

The boundary gap $\delta$ is then defined by
\begin{Definition}\label{def:gap}
    \begin{equation}\label{gap}
    \delta=\max_{F\in \facets} \big(\max_{x\in\Boundary} \vec{n}_{F}\cdot(x-\vec{g}_F)\big).
\end{equation}
\end{Definition}

\begin{Remark}
    $\delta \sim O(l) $ if $\Boundary$ is Lipschitz, and $\delta \sim O(l^2) $ if $\Boundary$ is $C^2$, where $l$ denotes the largest diameter of the facets of $\SubDomain$.
\end{Remark}

\subsection{Triangulation and Finite Element Space}
We consider Lagrange elements on simplexes. The general theory for finite element spaces can be found in \cite{Brenner-Scott}.

Let $\Triangulation$ be a conforming triangulation of $\SubDomain$ consisting of simplexes. For each $T\in \Triangulation$, we denote $h_T$ as the length of the largest edge, $\rho_T$ as the supremum of the diameters of balls contained in $T$, and $R_T$ as the circumsradius of $T$.

The mesh size $h$ is defined as
\begin{equation}\label{meshsize}
    h=\max_{T\in \Triangulation}h_T.
\end{equation}

The maximal circumradius of the mesh is defined by
\begin{equation}\label{def:Rh}
    R_h = \max_{T \in \Triangulation} R_T,
\end{equation}

$\Triangulation$ is said to be regular if there is a $\sigma>0$ such that $h_T/\rho_T\leq \sigma$ for each $T\in\Triangulation$.

For 2-dimensional cases, $\Triangulation$ is said to be non-blunt if each $T\in\Triangulation$ is non-blunt, the minimal angle $\theta_0$ of $\Triangulation$ is defined as the minimum of all interior angles of the elements $T\in\Triangulation$.

We denote the interior nodes of $\Triangulation$ as $x_1, x_2,\dots,x_M$, and the boundary nodes as $x_{M+1},\dots,x_N$, with corresponding nodal basis functions $\varphi_1,\dots,\varphi_N$.

Then the finite element spaces are defined by
\begin{equation}
    \begin{aligned}
        H^1(\SubDomain)\supset V_h=\text{Span}(\{\varphi_1,\dots,\varphi_N\}),\\
        H_{0}^1(\SubDomain)\supset V_{h,0}=\text{Span}(\{\varphi_1,\dots,\varphi_M\}).
    \end{aligned}
\end{equation}

For each $a\in C(\overline{\SubDomain})$, its interpolation onto $V_{h,0}$ is defined as
\begin{equation}
    \Pi_{h}a=\sum_{i=1}^{N} a(x_i)\varphi_i.
\end{equation}

Then if $a\in H^2(\SubDomain)$ the interpolation error can be estimated as
\begin{equation}
\begin{aligned}
    \|a-\Pi_h a\|_{L^2(\SubDomain)}\leq \big(\max_{T\in\Triangulation} E_0(T)\big) |a|_{H^2(\SubDomain)},\\
    |a-\Pi_h a|_{H^1(\SubDomain)}\leq \big(\max_{T\in\Triangulation} E_1(T)\big) |a|_{H^2(\SubDomain)}.
\end{aligned}
\end{equation}
Here we define
\begin{Definition}[Local interpolation error constants]\label{def:LocalErrorC}
    \begin{equation}
\begin{aligned}
    E_0(T)=\sup \{|v|_0;v\in H^2(T), |v|_2=1,v(p_i)=0\text{ for }i=1,\dots,n\},\\
    E_1(T)=\sup \{|v|_1;v\in H^2(T),|v|_2=1,v(p_i)=0\text{ for }i=1,\dots,n\},
\end{aligned}
\end{equation}
for each $T\in\Triangulation$, where $p_1,\dots,p_n$ are the vertices of $T$.
\end{Definition}

\subsection{Finite Element Discretization}

The finite element solution $u_h\in V_{h,0}$ of problem \eqref{ExactProblem} is defined by
\begin{equation}\label{FEM}
    \int_{\SubDomain}\nabla u_h\cdot\nabla v_h dx=\int_{\SubDomain} f_h v_h dx, \forall v_h\in V_{h,0}
\end{equation}
Here $f_h\in L^2(\SubDomain)$ is an approximation of $f$, and is typically taken as the interpolation of $f$ onto $V_{h}$, and is extended by zero outside $\SubDomain$.

\subsection{Numerical Solution}
The Finite Element Equation \eqref{FEM} can be expressed as the following linear algebraic system:
\begin{equation}
    Ax=b,
\end{equation}
where
\begin{equation}
    \begin{aligned}
        &A_{ij}=\int_{\SubDomain}\nabla\varphi_i\cdot\nabla\varphi_j dx\text{ for }i,j=1,\dots,M,\\
        &b_i=\int_{\SubDomain} f_h\varphi_idx\text{ for }i=1,\dots,M,\\
        &u_h=\sum_{j=1}^{M}x_j\varphi_j.
    \end{aligned}
\end{equation}

The solution of $Ax=b$ inevitably involves rounding errors arising from finite-precision arithmetic. A detailed analysis of such errors lies outside the scope of this paper; see, for example, Higham~\cite{higham2002accuracy} for a comprehensive treatment of numerical accuracy and stability in floating-point computations.

\section{Upper Bound of Consistency Error}
\begin{figure}[H]
    \centering
    \includegraphics[width=0.8\linewidth]{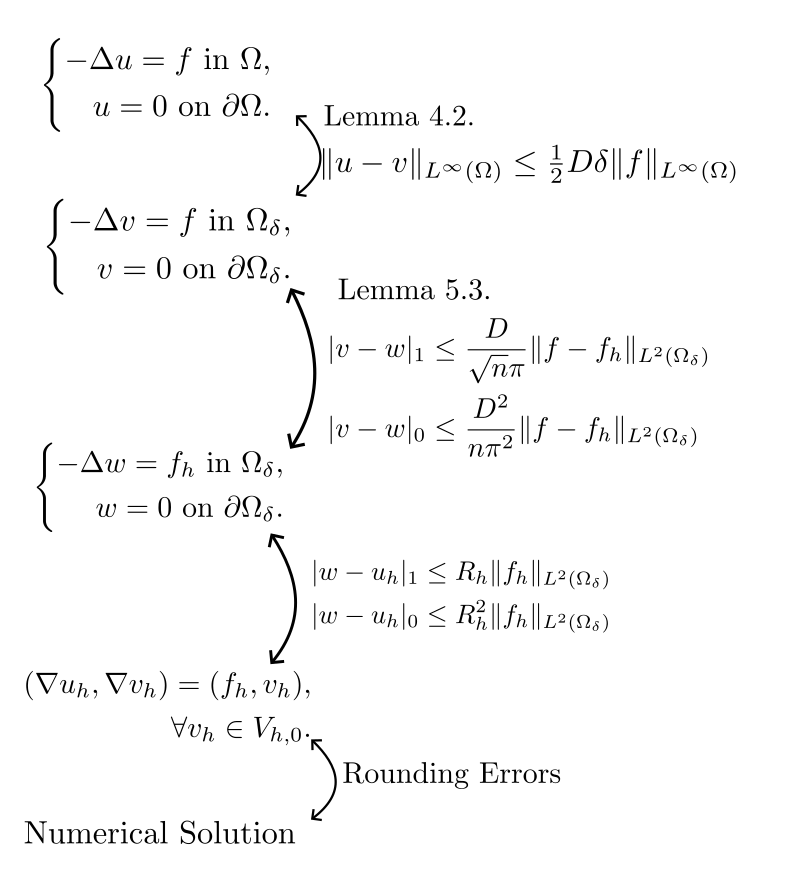}
    \caption{Hierarchical decomposition of the total consistency error.}
    \label{errorFig}
\end{figure}

In this section, we state the main results and the proof strategy.

To estimate the total consistency error, we introduce a sequence of auxiliary problems corresponding to different layers of approximation (see Figure \ref{errorFig}). Let $u,v,w$ denote the weak solutions of the corresponding boundary value problems shown in figure \ref{errorFig}, and let $u_h$ denote the finite element solution of the discrete problem. 

Each layer isolates one source of deviation between the exact and numerical solutions: geometric truncation, data projection, finite-element discretization, and rounding errors.

The errors arising from the approximation of the domain and the source term will be stated later in lemma \ref{lem: boundary},\ref{lem:force}.

Since $\SubDomain$ is convex and $f_h\in L^2(\Domain)$, it follows that $w\in H^2(\SubDomain)$. We can therefore apply the theorem 2.1 in \cite{Liu2024} and obtain
\begin{equation}\label{localEstimate}
    \left\{\begin{aligned}
        |w-u_h|_1\le \max_{T\in\Triangulation}\big(E_1(T)\big)\|f_h\|_{L^2 (\SubDomain)},\\
        |w-u_h|_0\le \max_{T\in\Triangulation}\big(E_1(T)^2\big)\|f_h\|_{L^2 (\SubDomain)}.
    \end{aligned}\right.
\end{equation}

Then from lemmas \ref{lem: boundary},\ref{lem:force}, and~\eqref{localEstimate}, together with the formulas 
for $E_1(T)$ proposed in \cite{KIKUCHI20073750,TetrahedronExplicit}, we obtain the following explicit estimate of total $L^2$ consistency error.

\begin{theorem}[Explicit Consistency Error]\label{thm:main}
    Let $u$ be the weak solution to \eqref{ExactProblem}, $u_h$ be the finite element solution solution to \eqref{FEM}. Then
    \begin{equation}
        |u-u_h|_0 \le \tfrac{1}{2}D\,|\Domain|^{\frac{1}{2}}\,\delta \,\maxForce+C_P(\SubDomain)^2\|f-f_h\|_{L^2(\SubDomain)}+A_h ^2\,\|f_h\|_{L^2(\SubDomain)}.
    \end{equation}
    Here, $C_P(\SubDomain)$ denotes the Poincaré constant of $\SubDomain$, defined as the inverse of the square root of the smallest eigenvalue of
    \begin{equation}
        \left\{\begin{aligned}
            -\Laplace u=\lambda u\text{ in }\SubDomain,\\
            u=0\text{ on }\SubBoundary,
        \end{aligned}\right.
    \end{equation}
    and it satisfies the geometric bound
\begin{equation}\label{eq:PoincareBound}
    C_P(\SubDomain) \le \frac{D}{\sqrt{n}\pi}.
\end{equation}
    Moreover,
    \begin{equation}
        A_h=\max_{T\in\Triangulation}\big(E_1(T)\big),
    \end{equation}
    which depends explicitly on the space dimension $n$ and the
global geometric properties of the mesh $\Triangulation$, as follows:
\begin{equation}
A_h \le
\begin{cases}
\textbf{For $n=2$:} &
\begin{cases}
R_h, & \text{(maximal circumradius, see~\eqref{def:Rh}),}\\[6pt]
0.69711\, \dfrac{\cos^2(\tfrac{\theta_0}{2})}{\sin(\tfrac{\theta_0}{2})}\, h,
    & \text{(if $\theta_0$ is the minimal angle),}\\[10pt]
\sqrt{\tfrac{11}{60}}\, h, & \text{(if $\Triangulation$ is non-blunt).}
\end{cases}\\[14pt]
\textbf{For $n=3$:} &
2.19\, \sigma \, h\quad \text{(if $h_T / \rho_T \le \sigma$ for all $T \in \Triangulation$).}
\end{cases}
\end{equation}
\end{theorem}

\begin{Remark}[Upper Bound of $\|f-f_h\|_{L^2(\SubDomain)}$]
    Suppose $f\in C^1(\overline{\Domain})$, and $f_h$ is defined as the piecewise constant barycenters interpolation of $f$ on the elements $T\in\Triangulation$,
    \begin{equation}
        f_h=\sum_{T\in\Triangulation}f(g_T)\chi_T,
    \end{equation}
    where $g_T$ denotes the barycenter of $T$ and $\chi_T$ is its characteristic function. Then the following explicit bound holds.
    \begin{equation}\label{ineq:C1}
        \|f-f_h\|_{L^2(\SubDomain)}\leq \tfrac{n}{n+1}h|\SubDomain|^{\frac{1}{2}}\|\nabla f\|_{L^\infty(\Domain)}.
    \end{equation}
    Indeed, for any $x\in T$,
    \begin{equation}
        |f(x)-f(g_T)|\leq \|\nabla f\|_{L^\infty(\Domain)}|x-g_T|\leq \tfrac{n}{n+1}h_T \|\nabla f\|_{L^\infty(\Domain)}, 
    \end{equation}
    hence 
    \begin{equation}
        \|f-f_h\|_{L^2(T)}\le\tfrac{n}{n+1}h_T |T|^{\frac{1}{2}}\|\nabla f\|_{L^\infty(\Domain)},
    \end{equation}
    for each $T\in\Triangulation$. Summing over $T\in\Triangulation$ yields \eqref{ineq:C1}.

    If, alternatively, $f\in H^2(\Domain)$, and $f_h=\Pi_h f$, then
    \begin{equation}
         \|f-f_h\|_{L^2(\SubDomain)}\leq \max_{T\in\Triangulation}\big(E_0(T)\big)|f|_2,
    \end{equation}
    where $E_0(T)$ is defined in \eqref{def:LocalErrorC}.

    Explicit formula for $E_0(T)$ can be found in theorem~1 in \cite{Arcangeli} and Theorem~1.1 in \cite{kobayashi2025}, which imply the following upper bound.
    \begin{equation}
        \max_{T\in\Triangulation}E_0(T) \le
        \begin{cases}
        \sqrt{\tfrac{3}{83}}\, h^2, & \text{for } n = 2,\\[6pt]
        8\, h^2, & \text{for } n = 3,
        \end{cases}
        \end{equation}
\end{Remark}

As an example, we obtain the following explicit consistency error estimate for a two-dimensional non-blunt mesh, where $f\in H^2(\Domain)$.
\begin{corollary}
    Let $\Domain\subset \mathbb{R}^2$ be a convex, bounded domain ,$f\in H^2(\Domain)$ and $f_h=\Pi_h f$. Let $u$ be the weak solution of \eqref{ExactProblem}, and let $u_h$ be the finite element solution defined as \eqref{FEM}, where the ttriangulation $\Triangulation$ is non-blunt. Then the following explicit $L^2$ error bound holds.
    \begin{align}
        |u-u_h|_0\leq \frac{1}{2}D|\Domain|^{\frac{1}{2}}\delta\|f\|_{L^\infty}+0.1834h^2|f|_0\notag\\
        +9.632\times 10^{-3}D^2 h^2|f|_2+3.486\times 10^{-2}h^4 |f|_2
    \end{align}
        
\end{corollary}

\section{Boundary Perturbation}

We evaluate the effect of boundary perturbation by deriving the upper bound of $|u-v|$ in the gap $\Domain\setminus \SubDomain$, and applying maximum value principle in $\SubDomain$.

First, we characterize the gap $\Domain\setminus \SubDomain$ using signed distance functions defined on the facets $F \in \facets$ of $\SubDomain$.
\begin{Definition}
    \begin{equation}
        p_F(x)=\vec{n}_F\cdot (x-\vec{g}_F)
    \end{equation}
    for all $x\in \EucSp$, where $\vec{n}_F$ denotes the outward unit normal vector on $F$, and $\vec{g}_F$ is the barycenter of $F$.
\end{Definition}

Then since $\SubDomain$ is convex,
\begin{equation}
\begin{aligned}
    &\SubDomain=\{x\in \EucSp:p_F(x)<0,\forall F\in \facets\},\\
    &\Domain\setminus \SubDomain=\cup_{F\in\facets}\{x\in \Domain:p_F(x)\ge 0\}
\end{aligned}
\end{equation}

\begin{lemma}\label{lem: boundary}
    Let $u$ and $v$ be the weak solutions to the boundary value problems illustrated in Figure~\ref{errorFig}. 
Then
\begin{equation}
    \|u-v\|_{L^\infty(\Domain)}\le \frac{1}{2}\,D\,\delta\,\maxForce.
\end{equation}
\end{lemma}

\begin{proof}
    \textbf{Step 1.}We derive an upper bound of $|u-v|$ in the gap $\Domain\setminus\SubDomain$.

    Let $F$ be a facet of $\SubDomain$, and $r_0$ be a point on $\Boundary$ where the signed distance function $p_F$ attains the maximum.

    After translation and rotation of the coordinate axes, we set 
    \begin{equation}
        x_0=0,\,\vec{n}_F=-\vec{e}_1.
    \end{equation}
    Then
    \begin{equation}
        p_F(x) = \delta_0 - x_1 \quad \text{for all } x \in \EucSp ,
    \end{equation}
    where $\delta_0$ is the first coordinate of $\vec{g}_F$ and satisfies
    $0 \le \delta_0 \le \delta$.
    
    Define the barrier function
    \begin{equation}
        U(x)=\frac{1}{2}\,\maxForce\,x_1\,(D-x_1),
    \end{equation}
    so that $-\Laplace U=\maxForce$.
    Since
    \begin{equation}
        0\le x_1\le D,\,\forall x\in \overline{\Domain},
    \end{equation}
    we have
    \begin{equation}
        U(x)\ge 0,\,\forall x\in \Boundary.
    \end{equation}
    Because $u=0$ on $\Boundary$, the comparison principle yields
    \begin{equation}
        |u(x)|\le U(x)\text{ in }\Domain.
    \end{equation}
    Thus, 
    \begin{equation}
        |u(x)|\le \frac{1}{2}\,D\,\delta\,\maxForce\text{ in }\{x\in\Domain:p_F(x)\ge 0\}.
    \end{equation}

    Applying this argument to each $F\in \facets$ gives
    \begin{equation}
        |u(x)|\le \frac{1}{2}\,D\,\delta\,\maxForce\text{ in }\cup_{F\in\facets}\{x\in\Domain:p_F(x)\ge 0\}=\Domain\setminus \SubDomain.
    \end{equation}

    Since $v$ is extended by 0 in $\Domain\setminus\SubDomain$, we have
    \begin{equation}\label{ineq:gap}
        |u(x)-v(x)|\le \frac{1}{2}\,D\,\delta\,\maxForce\text{ for }x\in \Domain\setminus\SubDomain.
    \end{equation}

    \textbf{Step 2.} We derive an upper bound of $|u-v|$ in $\SubDomain$.

    Because
    \begin{equation}
        -\Laplace (u-v)=0\text{ in }\SubDomain,
    \end{equation}
    the maximum value principle implies
    \begin{equation}
        \max_{\overline{\SubDomain}}|u-v|= \max_{\SubBoundary}|u-v|.
    \end{equation}
    By $\SubBoundary\subset \Domain\setminus\SubDomain$ and \eqref{ineq:gap},
    \begin{equation}
        \max_{\SubDomain}|u-v|\le \frac{1}{2}\,D\,\delta\,\maxForce.
    \end{equation}
    This completes the proof.
\end{proof}
\begin{Remark}[Optimality]
    The constant $\tfrac{1}{2}$ is at most $n$ times the optimal constant for general convex domains in $\EucSp$.
    Indeed, if $\Domain = B(0, \tfrac{D}{2})$ and $f \equiv 1$, 
    the exact norm is
    \begin{equation}
        \|u - v\|_{L^\infty(\Domain)}
            = \tfrac{1}{2n} \, \delta (D - \delta).
    \end{equation}
\end{Remark}

\section{Approximation of Source term}
We study the perturbation of the source term with a standard energy analysis and a simple yet nearly sharp estimate of the Poincaré constant that only depends on the space dimension and the diameter of the domain.

\begin{lemma}[Poincaré inequality]\label{lem:Poincare}
    Let $V\subset\EucSp$ be a bounded open domain with $\operatorname{diam}(V) \le D$. Then, for all $u \in H_0^1(V)$,
    \begin{equation}
        \|u\|_{L^2(V)}\,\le\,\frac{D}{\sqrt{n}\pi}|u|_{H^1(V)}.
    \end{equation}
\end{lemma}
\begin{proof}
    We translate the coordinate axis so that 
    \begin{equation}
        V\subset W:=(0,D)^n.
    \end{equation}
    Since $u \in H_0^1(V)$, its zero extension to $L$ also belongs to $H_0^1(W)$, and hence admits a Fourier sine expansion
    \begin{equation}\label{eq:Fourier}
        u(x) = (\frac{2}{D})^n\sum_{\alpha \in \mathbb{N}^n} c_\alpha 
    \prod_{i=1}^n \sin\!\left(\frac{\pi \alpha_i x_i}{D}\right),
    \end{equation}

    In addition,
    \begin{equation}\label{eq:FourierDer}
        \partial_k u(x)
    = \left(\frac{2}{D}\right)^n 
    \sum_{\alpha \in \mathbb{N}^n} 
    c_\alpha \, \frac{\pi \alpha_k}{D} 
    \cos\!\left(\frac{\pi \alpha_k x_k}{D}\right)
    \prod_{\substack{i=1 \\ i \ne k}}^{n} 
    \sin\!\left(\frac{\pi \alpha_i x_i}{D}\right).
    \end{equation}

Thus, 
\begin{equation}
    \begin{aligned}
        &\|u\|_{L^2(W)}^2=\sum_{\alpha \in \mathbb{N}^n}|c_\alpha|^2,\\
        &|u|_{H^1(W)}^2=(\frac{\pi}{D})^2\sum_{\alpha \in \mathbb{N}^n}|c_\alpha|^2(\sum_{k=1}^{n}|\alpha_k|^2).
    \end{aligned}
\end{equation}
It follows that
\begin{equation}
    \|u\|_{L^2(W)}^2\le \frac{D^2}{n\pi^2}|u|_{H^1(W)}^2.
\end{equation}
Finally, since $u$ is extended by 0 outside $V$,
\begin{equation}
    \|u\|_{L^2(V)}=\|u\|_{L^2(W)},\,|u|_{H^1(V)}=|u|_{H^1(W)}
\end{equation}
This completes the proof.
\end{proof}
\begin{Remark}[Optimality]
    The exact Poincaré constant is the inverse of the square root of the smallest Dirichlet eigenvalue of $-\Laplace$ on the domain $V$. The above estimate is nearly sharp for general bounded domains in $\mathbb{R}^n$, since the exact constants for the ball $B(0,\tfrac{D}{2})$ are approximately 
$0.208\,D$ when $n=2$, and exactly $\tfrac{D}{2\pi}$ when $n=3$.
\end{Remark}

\begin{lemma}\label{lem:force}
    Let $v$ and $w$ be the weak solutions to the boundary value problems illustrated in Figure~\ref{errorFig}. 
    Then 
    \begin{equation}
        \begin{aligned}
            &|v-w|_1\le \,C_P(\SubDomain)\|f-f_h\|_{L^2(\SubDomain)}\\
            &|v-w|_0\leq C_P(\SubDomain)^2 \|f-f_h\|_{L^2(\SubDomain)},
        \end{aligned}
    \end{equation}
    where the Poincaré constant satisfies
    \begin{equation}
        C_P(\SubDomain)\le \frac{D}{\sqrt{n}\pi}.
    \end{equation}
\end{lemma}
\begin{proof}
    We conduct a standard energy estimate. 
Subtracting the weak formulations for $v$ and $w$ yields
\[
-\Delta (v-w) = f - f_h \quad \text{in } \SubDomain, 
\qquad v-w = 0 \text{ on } \SubBoundary.
\]
Testing with $v-w$ and applying the Cauchy--Schwarz inequality gives
\[
|v-w|_{H^1(\SubDomain)}^2 
    \le \|f-f_h\|_{L^2(\SubDomain)} \|v-w\|_{L^2(\SubDomain)}.
\]
Using the Poincaré inequality for $v-w \in H_0^1(\SubDomain)$, we obtain
\[
|v-w|_{H^1(\SubDomain)} 
    \le C_P(\SubDomain)\, \|f-f_h\|_{L^2(\SubDomain)},
\]
and another application of the same inequality yields
\[
\|v-w\|_{L^2(\SubDomain)} 
    \le C_P(\SubDomain)^2\, \|f-f_h\|_{L^2(\SubDomain)}.
\]

Since $v$ and $w$ are extended by zero outside $\SubDomain$, it follows that
\begin{equation}
    |v - w|_1 = |v - w|_{H^1(\SubDomain)}, 
    \qquad 
    |v - w|_0 = \|v - w\|_{L^2(\SubDomain)}.
\end{equation}
By Lemma~\ref{lem:Poincare}, the Poincaré constant satisfies 
$C_P(\SubDomain) \le D / (\sqrt{n}\,\pi)$.

This completes the proof.
\end{proof}

\section{Total consistency Error}
We prove Theorem~\ref{thm:main} using Lemmas~\ref{lem: boundary} and~\ref{lem:force}, together with the interpolation error estimate.

\begin{proof}
    By isolating the sources of consistency error, we obtain
    \begin{equation}
    \begin{aligned}
        |u-u_h|_0\le &|u-v|_0+|v-w|_0+|w-u_h|_0\\
        \le &|\Domain|^{\frac{1}{2}}\|u-v\|_{L^\infty(\Domain)}+C_P(\SubDomain)^2\|f-f_h\|_{L^2(\SubDomain)}+|w-u_h|_0\\
        \le &\tfrac{1}{2}D\,|\Domain|^{\frac{1}{2}}\,\delta \,\maxForce+C_P(\SubDomain)^2\|f-f_h\|_{L^2(\SubDomain)}+|w-u_h|_0.
    \end{aligned}
    \end{equation}
    It remains to estimate the last term $|w - u_h|_0$.

    For completeness, we recall the standard interpolation estimate (see, e.g.\cite{Liu2024}) \eqref{localEstimate}.

    The Galerkin orthogonality gives
    \begin{equation}
        |w-u_h|_{H^1(\SubDomain)}\leq |w-\Pi_h w|_{H^1(\SubDomain)}.
    \end{equation}
    Using the local interpolation error estimate and the Miranda--Talenti inequality (\cite{MTineq}), we obtain
    \begin{equation}
        \begin{aligned}
            |w-\Pi_h w|_{H^1(\SubDomain)}^2=\sum_{T\in\Triangulation}|w-\Pi_h w|_{H^1(T)}^2\le \sum_{T\in\Triangulation} E_1(T)^2 |w|_{H^2(T)}^2\\
            \le A_h^2 |w|_{H^2(\SubDomain)}^2\le A_h^2 \|-\Laplace w\|_{L^2(\SubDomain)}^2= A_h^2 \|f_h\|_{L^2(\SubDomain)}^2..
        \end{aligned}
    \end{equation}
    This gives the $H^1$ seminorm estimate.
    
    Applying the Aubin–Nitsche duality (adjoint) argument then gives
    \begin{equation}
        \|w-u_h\|_{L^2(\SubDomain)}\leq A_h |w-u_h|_{H^1(\SubDomain)}\leq A_h^2 \|f_h\|_{L^2(\SubDomain)}^2,
    \end{equation}
    which is the $L^2$ estimate.

    Here we denote
    \begin{equation}
        A_h=\max_{T\in\Triangulation}E_1(T).
    \end{equation}

    For 2-dimensional triangular meshes, Liu~\cite{KIKUCHI20073750} proposed an explicit formula for the local interpolation constant $E_1(T)$ of a triangular element $T$.
    \begin{figure}[H]
        \centering
        \includegraphics[width=0.3\linewidth]{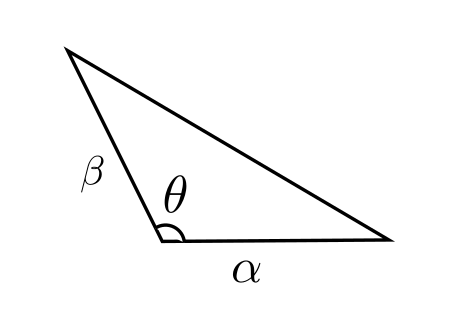}
        \caption{$\alpha$,$\beta$,$\theta$ of triangle T}
        \label{fig:FormulaLiu}
    \end{figure}
    \begin{equation}\label{formula:liu}
        E_1(T)\leq 0.49293\,\frac{1+|\cos\theta|}{\sin\theta}\,\sqrt{\frac{\alpha^2+\beta^2+\sqrt{\alpha^4+2\alpha^2 \beta^2\cos2\theta+\beta^4}}{2}},
    \end{equation}
    where $\alpha$ and $\beta$ denote the lengths of the two edges adjacent to the interior angle $\theta$.

    Besides, Kobayashi~\cite{kobayashi2025} proposed another explicit expression for the local interpolation constant
    \begin{equation}\label{formula:Kobayashi2}
    E_1(T)\le\sqrt{\frac{A^2 B^2 C^2}{16 S^2}
    - \frac{A^2 + B^2 + C^2}{30}
    - \frac{S^2}{5}\left(
    \frac{1}{A^2} + \frac{1}{B^2} + \frac{1}{C^2}
    \right)
    },
\end{equation}
    where $A,B$ and $C$ are the lengths of the three edges of the triangle, and $S$ denotes the area. This formula provides a considerably sharper estimate than~\eqref{formula:liu} for degenerate (highly acute or obtuse) triangles.

    For 3-dimensional tetrahedral meshes, Kobayashi\cite{TetrahedronExplicit} has proposed the explicit bound
    \begin{equation}\label{formula:Kobayashi3}
        E_1(T)\le 2.19\,\frac{\operatorname{diam}(T)^2}{\rho(T)}.
    \end{equation}
    
    An upper bound of $A_h$ can be obtained by element-wise evaluation of the above formulas \eqref{formula:liu},\eqref{formula:Kobayashi2},\eqref{formula:Kobayashi3}. Based on these expressions, we also derive explicit upper bounds of the global constant $A_h$ of typical simplicial meshes. For \textbf{3-dimensional} meshes, we present a bound involving regularity condition. For \textbf{2-dimensional} meshes, these bounds depend only on global geometric parameters of the mesh:
    The maximal element diameter $h$, the minimal interior angle $\theta_0$, and the maximal circumscribed radius $R_h$.
    In addition, we establish a bound depending solely on $h$ for meshes consisting of \emph{non-blunt} triangles,
    which is particularly relevant for ensuring the validity of the Discrete Maximum Principle.

    \paragraph{Evaluation by maximal circumradius in two dimensions.}
    It follows directly from~\eqref{formula:Kobayashi2} that
    \begin{equation}
    A_h \le R_h,
    \end{equation}
    where $R_h:=\max_{T\in\Triangulation}R_T$ is the maximal circumsradius of the mesh. Indeed, for each triangle $T\in\Triangulation$ with edge lengths $A,B,C$ and area $S$,
    \begin{equation}
        R_T=\frac{ABC}{4S}.
    \end{equation}
    \paragraph{Evaluation by mesh size and minimal angle in two dimensions.}
    Let $T\in\Triangulation$ be a triangular element, and let its longest edge be denoted as $e_1$. Then $|e_1|=h_T$, and is adjacent to the minimal angle of $T$, which we denote as $\theta$. Let $\beta\le h_T$ denote the length of the other edge adjacent to $\theta$.

    Substituting these notations into~\eqref{formula:liu} yields
    \begin{equation}
    \begin{aligned}
        E_1(T)\leq 0.49293\,h_T\frac{1+\cos\theta}{\sin\theta}\,\sqrt{\frac{2+\sqrt{2+2\cos2\theta}}{2}}\\
        \le0.49293\,h_T\frac{(1+\cos\theta)^\frac{3}{2}}{\sin\theta}\\
        =0.49293\,\sqrt{2}\,h_T \frac{\cos^2(\frac{\theta}{2})}{\sin(\frac{\theta}{2})}\\
        \le 0.69711\,h\,\frac{\cos^2(\frac{\theta_0}{2})}{\sin(\frac{\theta_0}{2})}.
    \end{aligned}
    \end{equation}
    This gives an upper bound of $A_h$ represented by mesh size and the minimal angle.

    \paragraph{Evaluation by mesh size for non-blunt meshes in two dimensions}
    We evaluate the interpolation constant by \eqref{formula:Kobayashi2}.
    Let $T\in\Triangulation$ be a triangular element, and let its longest edge be denoted as $e_1$, and the vertex opposite to $e_1$ be denoted as $P$.We assume that $T$ satisfies the non-blunt condition, namely that all interior angles are less than $\pi/2$.

    \textbf{case 1. $h_T=1$}

    We establish a coordinate system such that the endpoints of $e_1$ are located at $(-\tfrac{1}{2},0)$ and $(\tfrac{1}{2},0)$, and the vertex $P$ is placed at $(a,b)$ with $a\ge 0,b>0$. 

    Then since $h_T=1$, and $T$ is not blunt, $(a,b)$ satisfies
    \begin{equation}\label{eq:nonblunt_constraints}
        \begin{aligned}
            (a+\tfrac{1}{2})^2+b^2\le 1,\\
            (a-\tfrac{1}{2})^2+b^2\le 1,\\
            a^2+b^2\ge \tfrac{1}{4}.
        \end{aligned}
    \end{equation}

   \begin{figure}[H]
       \centering
       \includegraphics[width=0.5\linewidth]{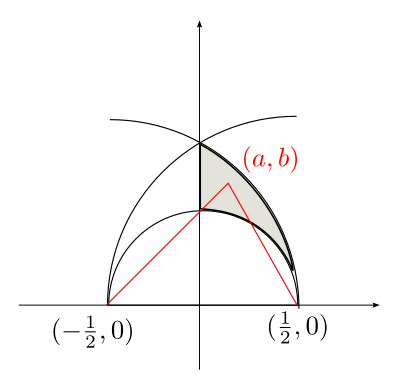}
       \caption{Constraints of $(a,b)$}
       \label{fig:constraints}
   \end{figure}
    
    Consequently, the three edge lengths of~$T$ are
    \begin{equation}
        \ell_1 = 1, \qquad
        \ell_2 = \sqrt{(a-\tfrac{1}{2})^2 + b^2}, \qquad
        \ell_3 = \sqrt{(a+\tfrac{1}{2})^2 + b^2},
    \end{equation}
    and the area of~$T$ is given by
    \begin{equation}
        |T| = \tfrac{1}{2}\,b.
    \end{equation}

    Substituting these values into \eqref{formula:Kobayashi2} gives
    \begin{equation}
        \begin{aligned}
            E_1(T)^2\le\frac{2}{15}b^2+\frac{13}{30}a^2+\frac{3}{40}+\frac{(\tfrac{1}{4}-a^2)^2}{4b^2}\\
            -\frac{1}{20}\frac{b^2}{(a+\tfrac{1}{2})^2+b^2}-\frac{1}{20}\frac{b^2}{(a-\tfrac{1}{2})^2+b^2}.
        \end{aligned}
    \end{equation}
    By \eqref{eq:nonblunt_constraints},
    \begin{equation}\label{ineq:relaxed}
        \begin{aligned}
            E_1(T)^2\le\frac{2}{15}b^2+\frac{13}{30}a^2+\frac{3}{40}+\frac{1}{4}(\frac{1}{4}-a^2)
            -\frac{1}{20}b^2-\frac{1}{20}b^2\\
            =\frac{1}{30}b^2+\frac{11}{60}a^2+\frac{11}{80}.
        \end{aligned}
    \end{equation}
    The maximum of the last row of \eqref{ineq:relaxed} can occur only along the boundary curve
    \begin{equation}
        \{(-\tfrac{1}{2}+\cos\theta,\sin\theta);0\le\theta\le\frac{\pi}{3}\}.
    \end{equation}
    Substituting this parametrization into the last \eqref{ineq:relaxed} yields
    \begin{equation}
        E_1(T)^2\le\max_{0\le\theta\le\tfrac{\pi}{3}} \big(\tfrac{13}{60}+\tfrac{3}{20}\cos^2\theta-\tfrac{11}{60}\cos\theta\big)=\frac{11}{60}
    \end{equation}
    Thus,
    \begin{equation}\label{ineq:normalizedTri}
        E_1(T)\le \sqrt{\frac{11}{60}},
    \end{equation}
    if $h_T=1$.

    \textbf{the general case.}
    It follows from \eqref{ineq:normalizedTri} and scaling arguments that
    \begin{equation}
        E_1(T)=h_TE_1(h_T^{-1}T)\le \sqrt{\frac{11}{60}}\,h_T\le \sqrt{\frac{11}{60}}\,h.
    \end{equation}
    This provides an explicit upper bound of~$A_h$ represented solely by the maximal mesh size~$h$
    for meshes consisting of non-blunt triangles.

    \paragraph{Evaluation by regularity and mesh size in three dimensions}
    It directly follows from \eqref{formula:Kobayashi3} that
    \begin{equation}
        A_h\le 2.19\,\sigma h,
    \end{equation}
    if the mesh satisfies
    \begin{equation}
        h_T/\rho_T\le\sigma,\forall T\in\Triangulation.
    \end{equation}

    We have thus completed the evaluation of the total consistency error
and derived explicit interpolation constants that depend only on global geometric parameters of the mesh.
\end{proof}

\section{Numerical Examples}
In this section, we present some numerical results to verify the validity of our error estimates. We consider the Poisson equation in the 2-dimensional unit disk
\begin{equation}\label{eq:test}
   \left\{ \begin{aligned}
        -\Laplace u=1\text{ in }B(0,1),\\
        u=0\text{ on }\partial B(0,1).
    \end{aligned}\right.
\end{equation}
The exact solution is given by
\begin{equation}
    u(x)=\tfrac{1}{4}(1-x^2).
\end{equation}
The problem is solved using \textbf{FreeFEM}.

We approximate the unit disk with a sequence of regular m-gons $\Omega_m$ with radius 1. For each polygonal domain $\Omega_m$, a triangulation $\mathcal{T}_m$ is generated with the Delauney algorithm implemented in \textbf{FreeFEM}.
The finite element discretization reads:
\begin{equation}
\int_{\Omega_m} \nabla u_m \cdot \nabla v_m dx
= \int_{\Omega_m} v_m dx,
\, \forall v_m \in V_m,
\end{equation}
where $V_m$ denotes the standard $P_1$ finite element space with homogeneous Dirichlet boundary condition. associated with $\mathcal{T}_m$.

Then theorem \ref{thm:main} predicts that
\begin{equation}\label{ineq:prediction}
    \|u-u_m\|_{L^2\!\big(B(0,1)\big)}\le \pi^\frac{1}{2} \left(A_m^2 + 2\,\sin^2\frac{\pi}{2m}\right),
\end{equation}
where $A_m$ will be computed with Kobayashi's formula \eqref{formula:Kobayashi2} after the generation of the mesh.

The actual $L^2$ error is given by
\begin{equation}\label{eq:actualError}
\begin{aligned}
    \|u-u_m\|_{L^2\!\big(B(0,1)\big)}^2=\int_{\Omega\setminus\Omega_m} |u(x)|dx +\int_{\Omega_m} |u(x)-u_m(x)|^2dx\\
    =\frac{m}{16}\int_{-\frac{\pi}{m}}^{\frac{\pi}{m}}d\theta\int_{\frac{\cos \frac{\pi}{m}}{\cos\theta}}^{1} (1-r^2)^2rdr+\int_{\Omega_m} |u(x)-u_m(x)|^2dx.
\end{aligned}
\end{equation}
The first term in \eqref{eq:actualError} is evaluated numerically in \textbf{Mathematica}, and the second term is computed numerically by integration over the mesh in \textbf{FreeFEM}.

The following table shows the numerically computed actual $L^2$ errors and the corresponding theoretical predictions for increasing values of $m$. Each numerical result is rounded upward and displayed with three significant digits after the decimal point.

\begin{table}[H]
\centering
\caption{Actual and Predicted Bound of $L^2$ Error}
\begin{tabular}{|c|c|c|c|c|c|c|c|c}
\hline
m  & 10  & 20  & 30  & 40  & 50 \\ \hline
actual  & 4.768e-2 & 1.303e-2  & 5.910e-3 & 3.248e-3 & 2.127e-3\\ \hline
predicted bound  & 2.397e-1  & 7.688e-2 & 4.368e-2 & 2.531e-2 & 1.672e-2 \\ \hline
\end{tabular}
\label{tab:example}
\end{table}
The predicted bound \eqref{ineq:prediction} based on Theorem \ref{thm:main} is valid and differs from the actual error by less than one order of magnitude.

The corresponding \textbf{FreeFEM} code is provide below.
\begin{lstlisting}[language=C++, caption={FreeFEM code}]
// approximation of the disk with a regular m-gon
int m=50;
real pi = 4*atan(1);
real f = 1.0;
border a(t=0, 2*pi){x=cos(t); y=sin(t); label=1;}

mesh disk = buildmesh(a(m));

// Computing Maximal Local Error Constant
real E2max = 0.0;

for (int k = 0; k < disk.nt; ++k) {
    // Triangle vertices
    real x0 = disk[k][0].x, y0 = disk[k][0].y;
    real x1 = disk[k][1].x, y1 = disk[k][1].y;
    real x2 = disk[k][2].x, y2 = disk[k][2].y;

    // Edge lengths
    real A2 =  (x1 - x2)^2 + (y1 - y2)^2 ;
    real B2 =  (x0 - x2)^2 + (y0 - y2)^2 ;
    real C2 =  (x0 - x1)^2 + (y0 - y1)^2 ;

    // Triangle area 
    real area = abs( disk[k].measure );
    real S2 = area * area;
    real E2 = (A2*B2*C2) / (16.0 * S2)
    -(A2+B2+C2)/30.0-(S2/5.0)*(1/A2+1/B2+1/C2);

    E2max = max(E2max, E2);
}

// Fespace
fespace femp1(disk, P1);
femp1 u, v;

problem laplace(u, v)
    = int2d(disk)(dx(u)*dx(v) + dy(u)*dy(v))
    - int2d(disk)(f*v)
    + on(1, u=0);
// Solve
laplace;

real prediction = sqrt(pi)*(E2max + 2*(sin(pi/(2*m)))^2);

// Internal Error^2
func err = u(x, y) - (1 - x^2 - y^2)/4;
real en=int2d(disk)( err^2 );

cout << "approx by a regular " << m << "-gon" << endl;
cout << "internal error energy= " << en << endl;
cout << "predicted error bound<= " << prediction << endl;
\end{lstlisting}

\bibliographystyle{IEEEtran}
\bibliography{refs}

\end{document}